\documentclass[twoside
]{amsart}%

\usepackage{amssymb}
\usepackage{amsmath}
\usepackage{graphicx}
  \usepackage{bbm}
  \usepackage{color}
\usepackage{amsthm}

\newcommand{\tend}[2]{\xrightarrow[#1\to#2]{}}

\def\E{\mathbb E}

\def\H{\mathbb H}

\newcommand{\tab}{{\mathbf{T}}_{\lambda}(a, b)}
\newcommand{\tkab}{{\mathbf{T}}_{\lambda_k}(a, b)}
\newcommand{\ttab}{{\mathbf{T}}_{\lambda_3}(a, b)}

\newcommand{\rkab}{{\mathbf{R}}_{k}(a, b)}
\newcommand{\rks}{{\mathbf{R}}_{k}(\ell_{s+1}, \ell_{s+2})}

\newtheorem{Th}{Theorem}
\newtheorem{Def}{Definition}
\newtheorem{Rem}{Remark}

\newtheorem{Prop}{Proposition}
\newtheorem{Cor}{Corollary}

\newtheorem{Lemme}{Lemma}

\author{\'Elise Janvresse, Beno\^it Rittaud, Thierry de la Rue}

\address{\'Elise Janvresse, Thierry de la Rue:
Laboratoire de Math\'ematiques Rapha\"el Salem, 
Universit\'e de Rouen, CNRS -- 
Avenue de l'Universit\'e -- 
F76801 Saint \'Etienne du Rouvray.}
\email{Elise.Janvresse@univ-rouen.fr\\Thierry.de-la-Rue@univ-rouen.fr}
\address{Beno\^it Rittaud: Laboratoire Analyse, G\'eom\'etrie et Applications, Universit\'e Paris 13 Institut Galil\'ee, CNRS -- 
99 avenue Jean-Baptiste Cl\'ement -- 
F93 430 Villetaneuse.\\
CNRS/Laboratoire de Math\'ematiques et Physique Th\'eorique, Universit\'e de Tours, UFR Sciences et Techniques -- Parc de Grandmont -- F37000 Tours}
\email{rittaud@math.univ-paris13.fr}

\title[Growth rate of random Fibonacci sequences]{Growth rate for the expected value of a generalized random Fibonacci sequence}

\begin{document}

\bibliographystyle{amsplain}
\maketitle

\begin{center}
{\bf Abstract}\\

\end{center}
A {\em random Fibonacci sequence} is defined by the relation $g_{n}=|g_{n-1}\pm g_{n-2}|$, 
where the $\pm$ sign is chosen by tossing a balanced coin for each $n$. 
We generalize these sequences to the case when the coin is unbalanced (denoting by $p$ 
the probability of a $+$), and the recurrence relation is of the form $g_{n}=|\lambda g_{n-1}\pm g_{n-2}|$.
When $\lambda \ge 2$ and $0<p\le 1$, we prove that the expected value of $g_{n}$ grows exponentially fast. 
When $\lambda=\lambda_{k}=2\cos(\pi/k)$ for some fixed integer $k\geq 3$, we show that the expected value of $g_{n}$ grows exponentially fast for $p>(2-\lambda_k)/4$ and give an algebraic expression for the growth rate.
The involved methods extend (and correct) those introduced in~\cite{rittaud2006}. 

\vspace{1cm}



\section{Introduction}
A {\em random Fibonacci sequence} is a sequence 
$(g_{n})_{n}$ defined 
by its first two terms $g_{1}$ and $g_{2}$ (which in the sequel are
assumed to be positive) and the recurrence relation
$g_{n+1}=\vert g_{n}\pm g_{n-1}\vert$, where for each $n$ the $\pm$ 
sign is chosen 
by tossing a balanced coin. 

A generalization of this notion consists in choosing the $\pm$ sign
by an unbalanced coin (say: $+$ with probability $p$ and $-$ with
probability $q:=1-p$). 
Another possible generalization consists
in fixing two real numbers, $\lambda$ and $\mu$, and considering 
 the recurrence relation $g_{n}=\vert \lambda g_{n-1}\pm \mu
g_{n-2}\vert$, 
where the $\pm$ sign is chosen by tossing a balanced
(or unbalanced) coin for each $n$. 
By considering the modified sequence
${\tilde g}_{n}:=g_{n}/\mu^{n/2}$, 
which satisfies ${\tilde g}_{n}=\vert \frac{\lambda}{\sqrt{\mu}} {\tilde g}_{n-1}\pm {\tilde g}_{n-2}\vert$,
we can always reduce to the case $\mu=1$.
In the following, we refer to the random 
sequences $g_{n}=\vert\lambda g_{n-1}\pm g_{n-2}\vert$ where the $+$
sign is chosen with probability $p$ and the $-$ sign with probability
$q:=1-p$ as {\em $(p,\lambda)$-random Fibonacci sequences}.

In~\cite{janvresse2007}, we investigated the question of the asymptotic growth
rate of almost all $(p,1)$-random Fibonacci sequences. 
We obtained an expression of this limit which is simpler than 
the one given by Divakar Viswanath in~\cite{viswanath2000} and which 
is not restricted to $p=1/2$: It is given by the integral of 
the natural logarithm over a specific measure (depending on $p$) 
defined on Stern-Brocot intervals.
The corresponding results for $(p,\lambda)$-random Fibonacci sequences, which
involve some techniques presented here but also complementary
considerations, is the object of another publication (see~\cite{janvresse2008b}).

\medskip
Here, we are concerned with the evaluation of the limit value of $(\E(g_{n}))^{1/n}$, 
where $(g_{n})_{n}$ is a $(p,\lambda)$-random Fibonacci 
sequence and $\E$ stands for the expectation. 

This limit was studied  in the particular case $p=1/2$ and $\lambda=1$ 
in \cite{rittaud2006}, where the growth rate of the expected value of a $(1/2, 1)$-random Fibonacci 
sequence is proved to be asymptotically equal to
$\alpha-1\approx 1.20556943$, where $\alpha$ is the only real zero of 
$\alpha^3=2\alpha^2+1$. The proof involves the study 
of the binary tree ${\mathbf{T}}$ naturally defined by the set of all 
$(1/2,1)$-random Fibonacci sequences (if a node $a$ has $b$ as a child, 
then the node $b$ has two children, labelled by $a + b$ and by $|a - b|$). 
The study of ${\mathbf{T}}$ is made by considering the biggest subtree of ${\mathbf{T}}$ (denoted by ${\mathbf{R}}$) which
shows no redundance, that is in which we never see two different edges
with the same values $a$ and $b$ (in this order) as  parent and 
child. 
This restricted tree ${\mathbf{R}}$ has many combinatorial and number-theoretic aspects which are of interest.
Let us mention that, after the publication of \cite{rittaud2006}, 
we realized that there was a problem in the last step of the proof of its main result; 
this mistake is explained in~Remark~\ref{Mistake}.

In the present paper, we correct and extend the result of~\cite{rittaud2006} to some 
$(p,\lambda)$-random Fibonacci sequences: For any $p\in[0, 1]$, any $\lambda\ge2$, and 
any $\lambda$ of the form $2\cos(\pi/k)$ (denoted by $\lambda_{k}$), where $k\geq 3$ is an integer. 

For $\lambda=\lambda_k$, the combinatorial properties of the tree ${\mathbf{R}}$ extends in a
surprinsingly elegant way, leading to an extremely natural
generalization of the results previously mentioned. In particular, the link made in~\cite{janvresse2007} and~\cite{rittaud2006} 
between the tree ${\mathbf{R}}$ and continued fraction expansion remains true for
$\lambda_{k}=2\cos(\pi/k)$ and corresponds to so-called Rosen
continued fractions, a notion introduced by David Rosen in~\cite{rosen1954}. We will not resort to continued fractions in the present work, but this aspect is presented in~\cite{janvresse2008b}.

An interesting fact to notice is that the values $\lambda_{k}$ and $\lambda>2$  are the only ones for which 
the group of transformations of the hyperbolic half plane $\H^2$ generated by the transformations $z\longmapsto -1/z$ and
$z\longmapsto z+\lambda$, said to be a {\em Hecke group}, is discrete (cf.~\cite{hecke1936}). 
We will not use that fact in the following, but it suggests that it 
is highly probable that some link is to be made between random Fibonacci sequences 
and hyperbolic geometry; in particular, possible future extensions of
the combinatorial point of view given by the restricted trees
defined below could have some interpretation in hyperbolic geometry
for values of $\lambda$ for which the corresponding M\"obius group is 
not discrete.

\medskip
We are grateful to Kevin Hare, of University of Waterloo (Canada),
for helpful comments, and to Jean-Fran\c{c}ois Quint, of CNRS and Universit\'e Paris-13 (France),
for fruitful discussions on hyperbolic geometry.

\section{Results}
Our main results are the following theorems.

\begin{Th}\label{Average1} 
Let $p\in [0,1]$, $k\ge 3$ and $m_{n}$ be the expected value of
the $n$-th term of a $(p,\lambda_{k})$-random Fibonacci sequence. 
\begin{itemize}
 \item If $p>p_c:=(2-\lambda_k)/4$, then
$$ \frac{m_{n+1}}{m_n}\ \tend{n}{\infty}\ \alpha_k(p)\left(1+\frac{pq^{k-1}}{\alpha_k(p)^k}\right) > 1, $$
where $\alpha_k(p)$ is the only positive root of the polynomial 
$$P_{k}(X):=X^{2k}-\lambda_{k}X^{2k-1}-(2p-1)X^{2k-2}-\lambda_k pq^{k-1}X^{k-1}-p^2q^{2k-2}.$$
\item If $p=p_c$, then $(m_{n})_{n}$ grows at most linearly.
\item If $p< p_c$, then $(m_{n})_{n}$ is bounded.
\end{itemize}
\end{Th}


Note that by Jensen's inequality, we have $\E[g_n]^{1/n} \ge\E\left[g_n^{1/n}\right]$. 
Hence the critical value for the growth rate of the expected value is smaller than the critical value when one considers the almost-sure growth rate. 
It is proved in~\cite{janvresse2008b} that the latter is equal to $1/k$, which is strictly larger than $p_c=(2-\lambda_k)/4=(1-\cos (\pi/k))/2$.

\begin{Th}
\label{lambda_grand}
Let $\lambda\ge2$, $0<p\le 1$, and $m_{n}$ be
the expected value of the $n$-th term of a $(p, \lambda)$-random
Fibonacci sequence. Then 
$$ \frac{m_{n+1}}{m_n}\ \tend{n}{\infty}\ \frac{\lambda + \sqrt{\lambda^2+4(2p-1)}}{2}. $$
\end{Th}

\medskip
In view of the study of $(p,\lambda)$-random Fibonacci sequences for
other values of $\lambda$, we also investigate an aspect of the
regularity of the behaviour of the growth rate of such a sequence in 
the neighbourhood of $\lambda=2$. 

\begin{Cor}\label{Cor:analyticite_p} 
Let $0<p\le 1$. 
If we assume that, for any $\lambda$ in the neighbourhood of 2, the expected value of a $(p,\lambda)$-random Fibonacci sequence increases exponentially fast with growth rate equal to ${\mathcal{G}}(\lambda)$, then ${\mathcal{G}}$ cannot be analytic at $\lambda =2$.
\end{Cor}

In the particular case $p=1/2$, we can even prove the non-analyticity of the growth rate on any left neighbourhood of $2$. 

\begin{Cor}\label{Cor:analyticite_1/2} 
Let $p=1/2$. Assume that, for any $\lambda\leq 2$, the expected value of a $(1/2,\lambda)$-random Fibonacci 
sequence increases exponentially fast, with growth rate equal to ${\mathcal{G}}(\lambda)$.
If ${\mathcal{G}}$ is differentiable at $\lambda=2$, then ${\mathcal{G}}'(2)=1$. 
If ${\mathcal{G}}$ is of class $C^n$ at $\lambda=2$, then ${\mathcal{G}}^{(i)}(2)=0$ for any $i\in [2,n]$. 
As a corollary, ${\mathcal{G}}$ cannot be analytic at $\lambda =2$.
\end{Cor}

\bigskip
We first prove Theorem~\ref{lambda_grand} in Section~\ref{Sec:lambda_grand}, since the case $\lambda\ge2$ is much easier.
Section~\ref{Combi} is devoted to general facts about the reduced tree for $\lambda=\lambda_k$.
It introduces notations and useful tools for the proof of Theorem~\ref{Average1}. 
This theorem is proved in Sections~\ref{Averagepi},~\ref{Averagepsi} and~\ref{end}, by extending the ideas introduced in \cite{rittaud2006} for the case $k=3$ and $p=1/2$. 
Proofs of corollaries~\ref{Cor:analyticite_p} and~\ref{Cor:analyticite_1/2} about the non-analyticity in the neighbourhood of $2$ can be found in Section~\ref{Sec:proof_of_cor}.
Section~\ref{Sec:open} contains a discussion about open questions.

\section{Case $\lambda\ge 2$}
\label{Sec:lambda_grand}

We introduce the tree $\tab$, which shows all different possible $(p,\lambda)$-random Fibonacci sequences with first positive terms $g_1=a$ and $g_2=b$: 
The root of the tree $\tab$ is labelled by $a$, its unique
child is labelled by $b$, and each node of the tree labelled by $\beta$
and with parent labelled by $\alpha$ has exactly two children, the right
one labelled by $\lambda\beta+\alpha$ and the left one by $|\lambda\beta-\alpha|$. 
When $\beta$ is the label of a node in $\tab$ with parent labelled by $\alpha$, 
it will be convenient to consider the vector $(\alpha, \beta)$ as the label of 
the corresponding edge. 
We also talk about the right and left children of an edge labelled $(\alpha, \beta)$, 
which are respectively the edges labelled by 
$(\beta, \lambda\beta + \alpha )$ and $(\beta, |\lambda\beta - \alpha |)$.

We also introduce the weight of an edge in the tree $\tab$, corresponding to the probability that a $(p, \lambda)$-random Fibonacci sequence passes through this edge. 
More formally, the initial edge in $\tab$ has weight $1$. 
If an edge has weight $w$, then its right and left children have respective weight $pw$ and $qw$. (Recall that $q=1-p$.)

We organize the edges of the tree $\tab$ in rows: 
The initial edge is the only edge in row 2; any child of an edge in row $n$ is in row $n+1$.
For any $n\geq 2$, we denote by $\psi_{n}$ the set of edges in row $n$ in $\tab$.

\begin{center}
\begin{figure}
\input{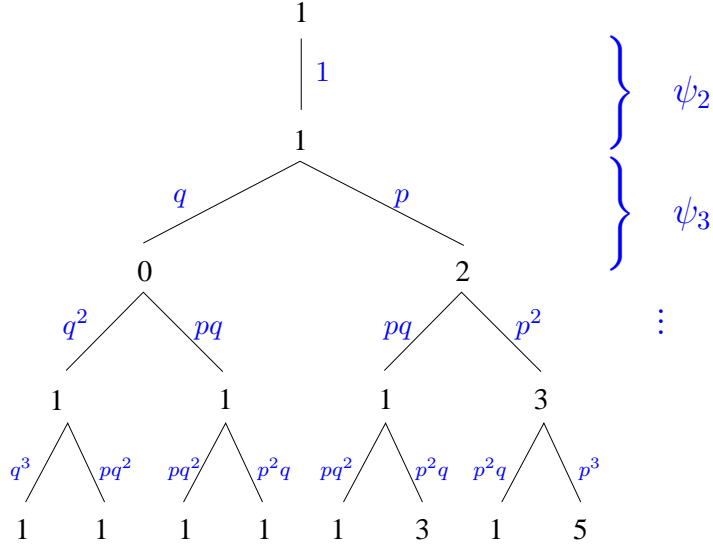}
\caption{The beginning of the tree ${\mathbf{T}}_{1}(1, 1)$.}
\end{figure} 
\end{center}

The {\em average} of any subset $X$ of edges of $\tab$ is defined as the sum $M(X)$ 
of all terms of the form $\beta w$, where $\beta$ is the second coordinate of an edge in $X$ 
and $w$ is the weight of this edge. 
Observe that the expected value $m_n$ of the $n$-th term of a $(p,\lambda)$-random Fibonacci is given by $M(\psi_n)$. 

It will be of interest to consider the sequence $(\ell_s)$ of numbers read along the leftmost branch of $\tab$: $\ell_1:=a$, $\ell_2:=b$ and $\ell_{s+1}:=| \lambda\ell_{s}-\ell_{s-1}|$ for $s\ge 2$.

\medskip
We now turn to the proof of Theorem~\ref{lambda_grand}.

\begin{Lemme}
\label{Lem:lineaire}
If $b\ge a$ and $\lambda\ge 2$, then for any edge in $\tab$ labelled by $(\alpha, \beta)$, we have 
$\beta\ge \alpha$.
\end{Lemme}

\begin{proof}
This is immediately proved by induction. 
\end{proof}

\begin{Cor}
 If $b\ge a$ and $\lambda\ge 2$, then $M(\psi_n) = \lambda M(\psi_{n-1}) + (2p-1)M(\psi_{n-2})$.
\end{Cor}

\begin{proof}
We deduce from Lemma~\ref{Lem:lineaire} that absolute values are never used in the computations of labels in $\tab$. 
Regrouping the contributions of the two children of an edge in row $n-1$, and summing over row $n-1$, we get the result.
\end{proof}

The preceding corollary shows that for $b\ge a$ and $\lambda\ge 2$, there exists constants $C$ and $C'$(depending on $a$ and $b$) such that 
$$
M(\psi_{n+2})=C \alpha^n + C' \alpha'^n,
$$
where $\alpha:=(\lambda+ \sqrt{\lambda^2+4(2p-1)})/2$ and $\alpha':=(\lambda - \sqrt{\lambda^2+4(2p-1)})/2$ are the roots of the polynomial $X^2-\lambda X -(2p-1)$. 
Observe that, by Lemma~\ref{Lem:lineaire}, if $b\ge a$ then all labels in the tree $\tab$ are larger than $b$, hence $M(\psi_{n+2})\ge b>0$. Since $\alpha >1$ and  $|\alpha'|<1$, we deduce that $C>0$, and Theorem~\ref{lambda_grand} is proved when $b\ge a$.

If $b<a$, we first consider the case where the sequence $(\ell_s)$ along the leftmost branch is unbounded. 
Then there exists a first $S\ge2$ such that $\ell_{S+1}\ge \ell_S$. 
We can then consider $\tab$ as the disjoint union of the trees ${\mathbf{T}}_{\lambda}(\ell_{s}, \lambda\ell_{s}+ \ell_{s-1})$, $2\le s\le S$, and the tree ${\mathbf{T}}_{\lambda}(\ell_{S}, \ell_{S+1})$. 
$M(\psi_{n+2})$ can be written as a convex combination of similar expressions in these trees. Since for each one of these trees, the labels of the first edge are well-ordered, Theorem~\ref{lambda_grand} is valid for them. A simple computation shows that the result extends to $\tab$.

It remains to consider the case where $(\ell_s)$ is bounded. We then consider $\tab$ as the disjoint union of the leftmost branch and infinitely many trees, namely the trees ${\mathbf{T}}_{\lambda}(\ell_{s}, \lambda\ell_{s}+ \ell_{s-1})$, $2\le s$, whose first-edge labels are well-ordered. For each $s$, there exists constants $C_s$ and $C'_s$ (depending on $\ell_{s-1}$ and $\ell_s$) such that 
$$
M(\psi_{n+2})=q^n \ell_{n+2} + \sum_{s=0}^{n-1} q^s p \left(C_s  \alpha^{n-s+1} + C'_s \alpha'^{n-s+1}\right).
$$
Using the fact that $C_s$ and $C'_s$ are bounded, we obtain that $M(\psi_{n+2})\sim K\alpha^{n}$ for some $K>0$, which ends the proof of Theorem~\ref{lambda_grand}.


\section{Reduced tree in the case $\lambda=\lambda_k$ for some $k\ge 3$}\label{Combi}

{F}rom now on, we fix an integer $k\geq 3$, and set $\lambda:=\lambda_{k}$. We keep the notations concerning the tree $\tab$ introduced in the preceding section.

We introduce the two matrices 
\begin{equation}
\label{matrices}
L :=
\begin{pmatrix}
0 & -1\\
1 & \lambda
\end{pmatrix}
\qquad \mbox{and}\qquad 
R :=
\begin{pmatrix}
0 & 1\\
1 & \lambda
\end{pmatrix}.
\end{equation}
Observe that the right child of an edge labelled $(\alpha, \beta)$ is labelled by 
$(\alpha, \beta)R$, 
and the left child is labelled by $(u, |v|)$ where 
$( u, v) = (\alpha, \beta) L$.

\begin{Def}
Any right child in a tree is said to be a {\em $0$-th left child}. 
For any integer $m>0$, a child in a tree is said to be an {\em $m$-th left child} iff it is the left
child of an $(m-1)$-th left child. 
\end{Def}

\begin{Prop}\label{Retour} 
Let $(\alpha, \beta)$ be the label of an edge in $\tkab$. 
The $(k-1)$-th left child of the right child of this edge is also labelled by $(\alpha, \beta)$.
\end{Prop}

\begin{proof}
An elementary calculation shows that $L=PDP^{-1}$, where 
$$ D:=\begin{pmatrix}
e^{i\pi/k} & 0          \\
0          & e^{-i\pi/k}
\end{pmatrix},
\quad
P:=\begin{pmatrix}
1   &  e^{i\pi/k}          \\
1   &  e^{-i\pi/k}
\end{pmatrix}.
$$
As a consequence, we get that for any integer $j$,
\begin{equation}
 \label{RtimesPowersOfL}
 RL^j =\dfrac{1}{\sin(\pi/k)}\begin{pmatrix}
 \sin\frac{j\pi}{k}      &  \sin\frac{(j+1)\pi}{k}      \\
 \sin\frac{(j+1)\pi}{k}  &  \sin\frac{(j+2)\pi}{k}
\end{pmatrix}.
\end{equation}
In particular, for $j\le k-2$, $RL^j$ has nonnegative entries, and
$$
RL^{k-1} = 
\begin{pmatrix}
1 & 0\\
0 & -1
\end{pmatrix}.
$$
Therefore, for $j\le k-2$, the $j$-th left child of the right child of the edge labelled by $(\alpha, \beta)$ is labelled by $(\alpha, \beta)RL^j$ and the $(k-1)$-th left child of the right child is labelled by $(\alpha, |-\beta|)$. 
\end{proof}

We now define the tree $\rkab$ as the subtree of $\tkab$ obtained by removing the left child of the initial edge and all the edges which are $(k-1)$-th left child. 

\begin{Prop}[Linearity in $\rkab$]\label{proprieteR} 
Whenever it exists, the left child of an edge in $\rkab$ labelled by $(\alpha, \beta)$ is labelled by $(\alpha, \beta) L$.
\end{Prop}
\begin{proof}
By definition, $\rkab$ contains only $j$-th left children, for $0\le j\le k-2$.
The proposition is a direct consequence of the fact that for such $j$'s, $RL^j$ has nonnegative entries.
\end{proof}

Considering $\rkab$ as a subtree of $\tkab$, any edge in $\rkab$ inherits its weight from $\tkab$.
As for $\tab$, we organize the edges of $\rkab$ in rows, and denote by $\pi_n\subset\psi_n$ the set of edges in row $n$ in $\rkab$ (see Figure~\ref{Fig:arbreR}).

\begin{center}
\begin{figure}
\input{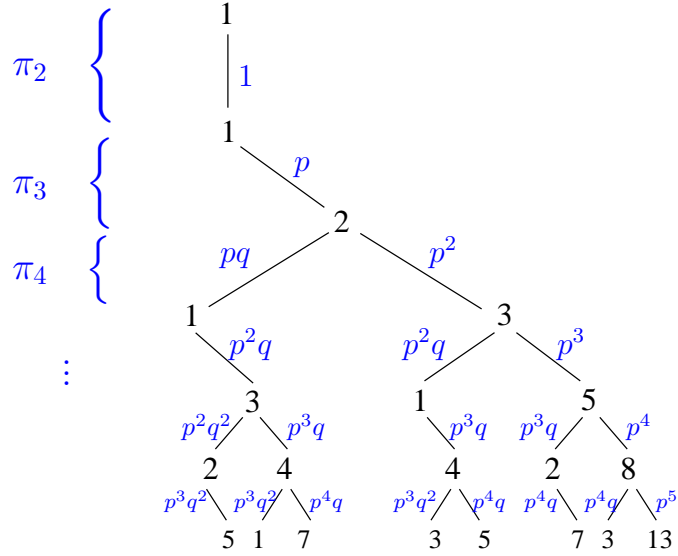}
\caption{The beginning of the tree ${\mathbf{R}}_{3}(1, 1)$, which is a subtree of ${\mathbf{T}}_{\lambda_3}(1, 1)$.}
\label{Fig:arbreR}
\end{figure} 
\end{center}

In Section~\ref{Averagepi}, we will first consider $M(\pi_n)$, which is easier to study than $M(\psi_n)$. 
Then, the next step (Section~\ref{Averagepsi}) will be to estimate $M(\psi_n)$ by partitioning the tree $\tkab$ in infinitely many copies of trees $\rks$, where $(\ell_s)$ is the sequence of numbers read along the leftmost branch of $\tkab$: $\ell_1:=a$, $\ell_2:=b$ and $\ell_{s+1}:=| \lambda_k\ell_{s}-\ell_{s-1}|$ for $s\ge 2$.

\begin{Rem}
\label{Mistake}
Let us explain why there is a mistake in the argument given in \cite{rittaud2006} (which deals with the particular case $k=3$ and $p=1/2$). 
The average growth rate is proved to be equal to some explicit value 
for two linearly independent pairs of initial values for the random Fibonacci sequence which are:
$(g_1,g_2)=(1,\varphi)$ and $(1,\varphi^{-1})$, where $\varphi=(1+\sqrt{5})/2$ is the golden ratio. 
It is then  asserted that, since the vectors $(1, \varphi)$ and 
$(1,\varphi^{-1})$ are linearly independent, any pair $(a,b)$ can 
be written as a linear combination of $(1, \varphi)$ and $(1,\varphi^{-1})$ to get the tree $\ttab$ 
as a linear combination of the trees ${\mathbf{T}}_{\lambda_3}(1,\varphi)$ and ${\mathbf{T}}_{\lambda_3}(1,\varphi^{-1})$;
The conclusion follows that, since both of these latter trees show 
the same growth rate, the growth rate of a random Fibonacci sequence does not
depend on the initial values $a$ and $b$ (with the obvious
restriction that $ab\neq 0$). 
In fact, this way of proving the theorem is wrong, 
as it can be easily seen by writing the equality
$(1,1)=\varphi^{-2}(1,\varphi)+\varphi^{-1}(1,\varphi^{-1})$.
The tree deduced from this linear combination is a tree with
root $1$ and child of the root equal to $1$, but it 
does not correspond to ${\mathbf{T}}_{\lambda_3}(1,1)$ (since, for example, 
the left grandchild of the root is equal to $2\varphi^{-3}$ instead of $0$).
\end{Rem}

\section{Average $M(\pi_n)$ of row $n$  of $\rkab$}
\label{Averagepi}
For any $X\subset \tkab$, and $0\le i\le k-2$, we define $X^i$ as the subset of $X$ made 
of all its elements which are $i$-th left children ($X^0$ thus corresponds to right children). 
With the convention $\pi_{2}^{k-2}:=\pi_{2}$ and $\pi_{2}^i=\emptyset$ for $i\le k-3$, we get that each node of the tree $\rkab$ which has only one child is a $(k-2)$-th left child.


\subsection{A recursive formula for $M(\pi_n)$}
The next lemma gives, for any $0\le i\le k-2$, $M(\pi_{n}^i)$ as a function of $(M(\pi_{n-1}^{j}))_{0\le j\le k-2}$ and $(M(\pi_{n-2}^{j}))_{0\le j\le k-2}$.

\begin{Lemme}\label{Moy} 
For any integer $n\geq 4$ we have 
\begin{eqnarray*}
M(\pi_{n}^0) &=& \lambda p M(\pi_{n-1}) + p M(\pi_{n-2}) -pq M(\pi_{n-2}^{k-2}),\\
M(\pi_{n}^1) &=& \lambda q M(\pi_{n-1}^0) - pq M(\pi_{n-2}),\\
M(\pi_{n}^i) &=& \lambda q M(\pi_{n-1}^{i-1}) -q^2 M(\pi_{n-2}^{i-2}), \quad 2\le i\le k-2.
\end{eqnarray*}
\end{Lemme}

\begin{proof} 
Consider an edge $e$ in $\pi_{n}$ whose parent has weight $w$ and is labelled by $(\alpha, \beta)$.

Case 1: Assume $e\in\pi_{n}^0$ ($e$ is a right child). The contribution of $e$ to $M(\pi_n^0)$ is 
$(\lambda\beta + \alpha) wp = \lambda p . \beta w + p . \alpha w$. 
Observe that when $e$ runs over $\pi_{n}^0$, its parent runs over $\pi_{n-1}$ and brings a contribution $\beta w$ to $M(\pi_{n-1})$. 
Moreover, when $e$'s parent runs over $\pi_{n-1}^0$, $e$'s grandparent runs over $\pi_{n-2}$ and has contribution $\alpha w/p$ to $M(\pi_{n-2})$. 
When $e$'s parent runs over $\cup_{i=1}^{k-2}\pi_{n-1}^i$, $e$'s grandparent runs over $\cup_{i=0}^{k-3}\pi_{n-2}^i$ and has contribution $\alpha w/q$ to $M(\cup_{i=0}^{k-3}\pi_{n-2}^i)=M(\pi_{n-2})-M(\pi_{n-2}^{k-2})$.
This proves the first equality of Lemma~\ref{Moy}.

Case 2: Assume $e\in\pi_{n}^1$ ($e$ is a left child). Its contribution to $M(\pi_n^1)$ is 
$(\lambda \beta - \alpha) wq = \lambda q . \beta w - pq . \alpha w/p$. 
Observe that $e$'s parent is in $\pi_{n-1}^{0}$ (thus is a right edge), and brings a contribution $\beta w$ to $M(\pi_{n-1}^{0})$.
Moreover $e$'s grandparent is in $\pi_{n-2}$, has weight $w/p$, and its contribution to $M(\pi_{n-2})$ is $\alpha w/p$.
When $e$ runs over $\pi_{n}^1$, its parent runs over $\pi_{n-1}^0$ and its grandparent runs over $\pi_{n-2}$.

Case 3: Assume $e\in\pi_{n}^i$ for $2\le i\le k-2$ ($e$  is a left child).
The contribution of $e$ to $M(\pi_n^i)$ is 
$(\lambda \beta - \alpha) wq = \lambda q . \beta w - q^2 . \alpha w/q$. 
Observe that $e$'s parent is in $\pi_{n-1}^{i-1}$ (thus is a left edge), and brings a contribution $\beta w$ to $M(\pi_{n-1}^{i-1})$.
Moreover $e$'s grandparent is in $\pi_{n-2}^{i-2}$, has weight $w/q$, and $e$ is its only left grandchild in $\pi_n^{i}$. The contribution of $e$'s grandparent to $M(\pi_{n-2}^{i-2})$ is $\alpha w/q$.
When $e$ runs over $\pi_{n}^i$, its parent runs over $\pi_{n-1}^{i-1}$ and its grandparent runs over $\pi_{n-2}^{i-2}$.
This ends the proof of the lemma.
\end{proof}

In the proof of the preceding lemma, we only used the structure of the tree $\rkab$ and the linear relation linking the labels of the edges in this tree, but not the specific value $\lambda=\lambda_k$. In the next lemmas, this specific value plays a central role.

\begin{Lemme}\label{proprieteRk}
We have for any $n\geq k+2$
$$ 
qM(\pi_{n-2}^{k-3}) - \lambda M(\pi_{n-1}^{k-2}) = pq^{k-2} M(\pi_{n-k})
$$
\end{Lemme}
\begin{proof}
Consider an edge $e$ in $\pi_{n-1}^{k-2}$ with weight $w$ and label $(\alpha, \beta)$. 
Its parent is in $\pi_{n-2}^{k-3}$ and has weight $w/q$: Thus, it contributes for $\alpha w/q$ to $M(\pi_{n-2}^{k-3})$. 
Moreover, $e$'s $(k-1)$-th parent is in $\pi_{n-k}$, has weight $w/pq^{k-2}$ and its label $(u, v)$ is such that $(u, v)RL^{k-2}=(\alpha, \beta)$. Thus, $(u, v)=(\beta, \alpha -\lambda\beta)$ and this edge contributes for $(\alpha -\lambda\beta)w/pq^{k-2}$ to $M(\pi_{n-k})$.
The conclusion follows from the fact that when $e$ runs over $\pi_{n-1}^{k-2}$, its parent runs over $\pi_{n-2}^{k-3}$ and its $(k-1)$-th parent runs over $\pi_{n-k}$.
\end{proof}

\begin{Lemme}\label{proprieteRk2}
We have for any $n\geq k+2$
$$ 
M(\pi_{n}^{k-2}) = pq^{k-2} \bigl(M(\pi_{n-k}) - q M(\pi_{n-k}^{k-2}) \bigr)
$$
\end{Lemme}
\begin{proof}
Consider an edge $e$ in $\pi_{n-k}$ with weight $w$ and label $(\alpha, \beta)$. 
If $e\in \pi_{n-k}^{j}$, $j\le k-3$, this edge is the ancestor of two edges in $\pi_{n}^{k-2}$, which have labels $(\alpha, \beta)RRL^{k-2}$ and $(\alpha, \beta)LRL^{k-2}$, and weights $wp^2q^{k-2}$ and $wqpq^{k-2}$. 
Their contribution to $M(\pi_{n}^{k-2})$ is thus $\beta wpq^{k-2}$.
If $e\in \pi_{n-k}^{k-2}$, it is the ancestor of only one edge in $\pi_{n}^{k-2}$, having labels $(\alpha, \beta)RRL^{k-2}$  and weight $wp^2q^{k-2}$. Its contribution to $M(\pi_{n}^{k-2})$ is thus $\beta wp^2q^{k-2}$.
The conclusion follows from the fact that any edge in $\pi_{n}^{k-2}$ has a unique ancestor in $\pi_{n-k}$.
\end{proof}

\begin{Lemme}\label{MoyPi}
We have for any $n\geq k+2$
$$
M(\pi_n) = \lambda M(\pi_{n-1}) + (2p-1)M(\pi_{n-2}) + pq^{k-1}M(\pi_{n-k}) + (q-p)q M(\pi_{n-2}^{k-2}).
$$
\end{Lemme}
\begin{proof}
Using Lemma~\ref{Moy} and Lemma~\ref{proprieteRk}, we get
\begin{eqnarray*}
 M(\pi_n) 
&=& \sum_{i=0}^{k-2}M(\pi_n^i) \\
&=& \left(\lambda p M(\pi_{n-1}) + p M(\pi_{n-2}) -pq M(\pi_{n-2}^{k-2})\right) \\
&& + \left(\lambda q M(\pi_{n-1}^0) - pq M(\pi_{n-2})\right)\\
&& + \sum_{i=1}^{k-3} \lambda q M(\pi_{n-1}^{i}) - \sum_{i=0}^{k-4} q^2 M(\pi_{n-2}^{i})\\
&=& \lambda M(\pi_{n-1}) - \lambda q M(\pi_{n-1}^{k-2}) +(p-pq-q^2)M(\pi_{n-2}) \\
&& + (q^2-pq) M(\pi_{n-2}^{k-2}) + q^2 M(\pi_{n-2}^{k-3})\\
&=& \lambda M(\pi_{n-1}) + (2p-1)M(\pi_{n-2}) + (q-p)q M(\pi_{n-2}^{k-2})+ pq^{k-1}M(\pi_{n-k}) .
\end{eqnarray*}
\end{proof}

In the particular case $p=q=1/2$, Lemma~\ref{MoyPi} gives $M(\pi_n)$ as a function of $M(\pi_{n-1})$ and $M(\pi_{n-k})$. 
The next proposition gives a simple way to express $M(\pi_{n})$ in terms of $(M(\pi_{m}))_{m<n}$ in the general case. 

\begin{Prop}\label{recurrenceMn}
We have for any $n\geq 2k+2$
\begin{multline}
M(\pi_n) = \lambda M(\pi_{n-1}) + (2p-1)M(\pi_{n-2}) + pq^{k-1}\lambda M(\pi_{n-k-1})\\  + p^2q^{2k-2}M(\pi_{n-2k})
\end{multline}
\end{Prop}

\begin{proof}
Applying twice Lemma~\ref{MoyPi}, (once for $M(\pi_n)$, then for $M(\pi_{n-k})$), the desired result follows from Lemma~\ref{proprieteRk2}. 
\end{proof}

\subsection{Study of $M(\pi_n)$}
\begin{Lemme}\label{Poly} 
Let $n\ge 2$ and let $Q(X)=X^n-\sum_{j=0}^{n-1}c_jX^j$, where $c_j>0$ for all $0\le j\le n-1$.
Then $Q$ has a unique positive real root $\alpha$. Moreover, $\alpha$ is of multiplicity $1$ and all the other roots of $Q$ have modulus strictly less than $\alpha$.

\end{Lemme}

\begin{proof} 
Since $Q(0) = -c_0<0$, $Q$ has at least one positive real root $\alpha$. 
Assume $\beta e^{i\theta}\not=\alpha$ is another root of $Q$ such that $\beta \ge \alpha$. 
We have 
$$
\beta^n = \Big|\sum_{j=0}^{n-1}c_j \beta^j e^{ij\theta}\Big| \le \sum_{j=0}^{n-1}c_j \beta^j.
$$
On the other hand, 
$$
\beta^n 
=\left(\frac{\beta}{\alpha}\right)^n\alpha^n 
= \sum_{j=0}^{n-1}c_j \beta^j \left(\frac{\beta}{\alpha}\right)^{n-j}
$$
which gives a contradiction if $\beta > \alpha$. 
Now, if $\beta=\alpha$, we proved that $\Big|\sum_{j=0}^{n-1}c_j \beta^j e^{ij\theta}\Big| = \sum_{j=0}^{n-1}c_j \beta^j$.
Since $c_j>0$ for all $0\le j\le n-1$, this implies $\theta=2\ell\pi$ and contradicts the fact that $\beta e^{i\theta}\not=\alpha$.

If $Q(\alpha)=Q'(\alpha)=0$, then $n \sum_{j=0}^{n-1}c_j \alpha^j =n\alpha^n = \sum_{j=0}^{n-1}j c_j \alpha^j $, which is impossible since $c_j>0$ for all $0\le j\le n-1$. Therefore, $\alpha$ is a simple root.
\end{proof}

\begin{Lemme}\label{Polyk} Let $k\geq 3$ be fixed and $\lambda = \lambda_k$. The polynomial
$$P_{k}(X):=X^{2k}-\lambda X^{2k-1} + (1-2p)X^{2k-2} - \lambda pq^{k-1}X^{k-1} - p^2q^{2k-2}$$ 
has a unique positive real root, denoted by $\alpha_{k}$, which is of multiplicity $1$.

Moreover, if $p>p_c$, all the other roots have modulus strictly less than $\alpha_{k}$.
If $p\le p_c$, $P_k$ has 2 conjugate roots of modulus $q$, its positive root $\alpha_k$ is smaller than $q$, and all the other roots have modulus strictly less than $\alpha_{k}$.
\end{Lemme}

\begin{proof} 
We claim that $P_{k}(X)$ can be rewritten as the product of $X^2-q\lambda X +q^2$ and 
$$
X^{2k-2} - pa_{2k-3}X^{2k-3} - p\sum_{j=0}^{k-3}a_{k-1+j}q^{k-3-j}X^{k-1+j} -p^2\sum_{j=0}^{k-2}a_{j}q^{2k-4-j}X^{j}
$$
where all coefficients $(a_j)_{0\le j\le 2k-3}$ are positive. 
Indeed, identifying terms of degree $j$ for $0\le j\le 2k-3$ yields to the following relations, from which we can inductively compute the coefficients $a_j$: 
\begin{eqnarray*}
a_0 &=& 1          \\
a_1 &=& \lambda  \\
a_j &=& \lambda a_{j-1}-a_{j-2} \qquad  \qquad \forall 2\le j\le k-2\\
a_{k-1} &=& \lambda + p(\lambda a_{k-2}-a_{k-3})\\
a_k &=& \lambda a_{k-1}-pa_{k-2} \\
a_j &=& \lambda a_{j-1}-a_{j-2} \qquad  \qquad \forall k+1\le j\le 2k-4\\
qa_{2k-3} &=& \lambda a_{2k-4}-a_{2k-5}.
\end{eqnarray*}
Observe that for $0\le i\le k-3$, $(a_i, a_{i+1})=(1, \lambda)L^i=(0, 1)RL^i$. 
Hence, we deduce from~\eqref{RtimesPowersOfL} that $a_i>0$ for all $0\le i\le k-2$. 
Note that~\eqref{RtimesPowersOfL} also gives
\begin{equation}
 \label{RLk-2}
RL^{k-2} = 
\begin{pmatrix}
\lambda & 1\\
1 & 0
\end{pmatrix}.
\end{equation}
This allows us to write $(a_{k-3}, a_{k-2})L=(0, 1)RL^{k-2}=(1,0)$, from which we get $a_{k-2}=1$ and $\lambda a_{k-2}-a_{k-3}=0$. 
Hence $a_{k-1} = \lambda > 0$, and then $a_k=\lambda^2-p>0$.
Now, we have, for $1\le i\le k-3$, $(a_{k-2+i}, a_{k-1+i})=(p, \lambda)L^i=(q\lambda , p)RL^i$. 
Thus, using~\eqref{RtimesPowersOfL} once again, $a_j>0$ for all $j\le 2k-4$.  Next, using~\eqref{RLk-2}, we can write 
$$(a_{2k-5}, a_{2k-4})L=(q\lambda , p)RL^{k-2}=(q\lambda^2+p,q\lambda), $$
and we get $\lambda a_{2k-4}-a_{2k-5}=q\lambda$. Hence we obtain $a_{2k-3}= \lambda > 0$, which proves that all the coefficients $a_j$ are positive.

{F}rom the same equation above, we also get $a_{2k-4}=q\lambda^2+p$, and it is then a simple computation to check that, in the product, terms of respective degree $2k-2$, $2k-1$ and $2k$ also coincide with those of $P_k$. 

\smallskip
Since $X^2-q\lambda X +q^2$ has two conjugate (nonreal) roots of modulus $q$, we conclude by using Lemma~\ref{Poly} that $P_k$ has a unique real positive root $\alpha_k$, which is of multiplicity~$1$. 
Moreover, since $P_{k}(q)= q^{2k-2}(2-\lambda -4p)$, we have $\alpha_k>q$ if and only if $p>(2-\lambda)/4$. 
\end{proof}

Let us denote by $(\beta_{j})$ the roots of the polynomial $P_k$ with $\alpha_k=\beta_0$, and let $\beta_1$ and $\beta_2=\overline{\beta_1}$ be the two conjugate (nonreal) roots of $X^2-q\lambda_k X +q^2$.
Thanks to Proposition~\ref{recurrenceMn}, for any $n\ge 0$, we have
\begin{equation}
 \label{equMpik}
M(\pi_{n+2}) = Q_{0}\beta_{0}^n + \sum_{j\not=0} Q_j(n)\beta_{j}^n,
\end{equation}
where $Q_j$ is a polynomial (depending on $a$ and $b$) of degree strictly less than the multiplicity of the root $\beta_j$.

Moreover, since $M(\pi_{n+2})$ is a real number, the coefficient $Q_0$ is real and coefficients corresponding to pairwise conjugate roots are conjugate.

\begin{Lemme}
\label{C0isPositive}
As soon as $(a,b)\not=(0,0)$, the coefficient $Q_0$ is positive. 
If $p\le p_c$, the roots of $X^2-q\lambda_k X +q^2$ are simple roots of $P_k$. 
If $p< p_c$, the associated coefficients $Q_1$ and $Q_2$ are null.
\end{Lemme}
\begin{proof}
Assume first that $p\le p_c$. 
By Lemma~\ref{Polyk}, the roots of $P_k/(X^2-q\lambda_k X +q^2)$ are of modulus smaller than $q$.
Hence $\beta_1=qe^{i\theta}$ and ${\beta_2}=qe^{-i\theta}$ (with $\theta\neq 0$ mod. $2\pi$) are simple roots of $P_k$. 
We also know that the associated coefficients are conjugate and that $\beta_{0}<q$ when $p< p_c$. 
Hence, we deduce from~\eqref{equMpik} that when $p< p_c$, $Q_1\neq0$ implies
$$ M(\pi_{n+2})\sim 2q^n \Re e(Q_1e^{in\theta}). $$
But this contradicts the fact that $M(\pi_{n+2})$ is positive for all $n$. 

If $p \ge p_c$, $\beta_0$ is the root of largest modulus of $P_k$. 
Moreover, since the coefficients $Q_1$ and $Q_2$ are null whenever $p< p_c$, we have 
$ M(\pi_{n+2})\sim C_{0}\beta_{0}^n$ whenever $p\not=p_c$. 
If $p=p_c$, we have $\beta_0=q$, thus 
$$ M(\pi_{n+2})\sim q^n \left( Q_0 + 2\Re e(Q_1(n)e^{in\theta}) \right). $$
If $(a,b)\not=(0,0)$, since $M(\pi_{n+2})$ is positive for all $n\ge1$, we conclude that $Q_0>0$ whatever $p$ is. 
\end{proof}

\section{Average $M(\psi_n)$ of row $n$ of $\tkab$, for $\lambda=\lambda_k$, $k\ge 3$} 
\label{Averagepsi}
\subsection{Left-most branch of $\tkab$}
We denote by $(\ell_{n})_{n}$ the sequence corresponding to the ``left
branch'' of the tree $\tkab$, that is, the sequence defined as
\[\ell_{1}:=a\quad \ell_{2}:=b\quad
\ell_{n}:=|\lambda\ell_{n-1}-\ell_{n-2}|\
(n\geq 3).\]

We have the following result:

\begin{Prop}\label{EllBornee}
The sequence $(\ell_{n})_{n}$ is bounded.
\end{Prop}

When $\lambda_k=1$ ($k=3$), it is easily seen that the sequence is
upper-bounded by $\max(a,b)$. Strangely enough, the general case is
quite more difficult to apprehend.
Note that we cannot use Proposition~\ref{proprieteR} and $L^k=Id$ to say that this sequence
is periodic (of period $k$) and thus bounded, since the $\ell_n$'s
belong to the left branch of the full tree $\tkab$, which is not contained in $\rkab$.

We give here a proof based on a geometrical interpretation, which can be applied for any $0<\lambda <2$. This proof also appears in~\cite{janvresse2008b}.
The key argument relies on the following observation: Let $\theta$ be such that $\lambda=2\cos\theta$. 
Fix two points $P_0,P_1$ on a circle centered at the origin $O$, such that the oriented angle $(OP_0,OP_1)$ equals $\theta$. 
Let $P_2$ be the image of $P_1$ by the rotation of angle $\theta$ and center $O$. Then the respective abscissae $x_0$, $x_1$ and $x_2$ of $P_0$, $P_1$ and $P_2$ satisfy $x_2=\lambda x_1 - x_0$. 
We can then geometrically interpret the sequence $(\ell_n)$ as the successive abscissae of points in the plane. 

\begin{Lemme}[Existence of the circle]
  Let $\theta\in ]0,\pi[$. For any choice of $(x, x')\in\mathbb{R}_+^2\setminus\{(0,0)\}$, their exist a unique $R>0$ and two points $M$ and $M'$, with respective abscissae $x$ and $x'$, lying on the circle with radius $R$ centered at the origin, such that the oriented angle $(OM,OM')$ equals~$\theta$.
\end{Lemme}

\begin{proof}
Assume that $x>0$. We have to show the existence of a unique $R$ and a unique $t\in ]-\pi/2,\pi/2[$ (which represents the argument of $M$) such that 
$$ R\cos t=x\quad\mbox{and}\quad R\cos(t+\theta)=x'. $$
This is equivalent to 
$$R\cos t=x \quad\mbox{and}\quad  \cos\theta - \tan t\ \sin \theta = \dfrac{x'}{x},$$
which obviously has a unique solution since $\sin \theta\neq 0$.

If $x=0$, the unique solution is clearly  $R=x'/\cos(\theta-\pi/2)$ and $t=-\pi/2$.

\emph{Remark:} Since $x_1>0$, we have $t+\theta<\pi/2$.
\end{proof}

\begin{proof}[Proof of Proposition~\ref{EllBornee}]
At step $n$, we interpret $\ell_{n+1}$ in the following way: Applying the lemma with $x=\ell_{n-1}$ and $x'=\ell_n$, we find a circle of radius $R_n>0$ centered at the origin and two points $M$ and $M'$ on this circle with abscissae $x$ and $x'$. Consider the image of $M'$ by the rotation of angle $\theta$ and center $O$. If its abscissa is nonnegative, it is equal to $\ell_{n+1}$, and we will have $R_{n+1}=R_n$. Otherwise, we have to apply also the symmetry with respect to the origin to get a point with abscissa $\ell_{n+1}$. The circle at step $n+1$ may then have a different radius, but we now show that the radius always decreases (see Figure~\ref{Fig:cercle}).

\begin{figure}\begin{center}
\input{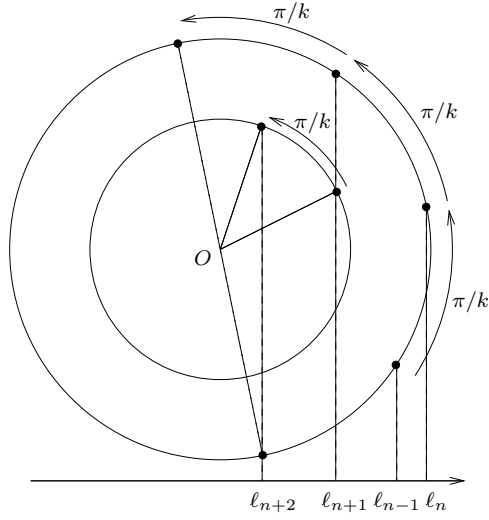}
\caption{$R_{n}=R_{n+1}$ is the radius of the largest circle, and $R_{n+2}$ is the radius of the smallest.}
\label{Fig:cercle}                  
\end{center}
\end{figure} 

Indeed, denoting by $\alpha$ the argument of $M'$, we have in the latter case $\pi/2-\theta<\alpha\le \pi/2$, $\ell_{n}=R_n\cos \alpha$ and $\ell_{n+1}=R_n\cos (\alpha+\theta+\pi)> 0$.
At step $n+1$, we apply the lemma with $x=R_n\cos \alpha$ and $x'=R_n\cos (\alpha+\theta+\pi)$. 
{F}rom the proof of the lemma, if $\ell_{n}=0$ (\textit{i.e.} if $\alpha=\pi/2$), $R_{n+1}=R_n\cos (\alpha+\theta+\pi)/\cos(\theta-\pi/2)=R_n$. 
If $\ell_{n}>0$, we have $R_{n+1}=R_n\cos \alpha/\cos t$, where $t$ is given by 
$$
\cos\theta - \tan t\ \sin \theta = \dfrac{\cos (\alpha+\theta+\pi)}{\cos \alpha} = -(\cos\theta - \tan \alpha\ \sin \theta).
$$
We deduce from the preceding formula that $\tan t + \tan\alpha = 2\cos\theta/\sin\theta>0$, which implies $t>-\alpha$. 
On the other hand, as noticed at the end of the proof of the preceding lemma, $t+\theta<\pi/2$, hence $t<\alpha$.
Therefore, $\cos \alpha<\cos t$ and $R_{n+1}<R_n$.

Since $\ell_n\le R_n\le R_1$ for all $n$, the proposition is proved. 
\end{proof}

\begin{Rem}
The behaviour of the sequence $(\ell_n)$ strongly depends on the initial values $a$ and $b$. 
It is proved in~\cite{janvresse2008b} that if $a/b$ admits a finite $\lambda$-continued fraction expansion, the sequence $(\ell_n)$ is ultimately periodic. On the other hand, when $\ell_n$ decreases exponentially fast to $0$, the exponent depends on the ratio $a/b$. 
Different examples of such behaviour are given in~\cite{janvresse2008b}.
\end{Rem}

\subsection{A formula for $M(\psi_n)$}

For any $s\geq 0$ and $n\ge 2$, we denote row $n$ of the tree $\rks$ by $\pi_{n, s}$.
In particular, $\pi_{n, 0}= \pi_{n}$ is row $n$ of the tree $\rkab$.

\begin{Prop}\label{SuccNu}
For any $n\geq 0$,
$$M(\psi_{n+2}) =
\sum_{m=0}^{\lfloor n/k\rfloor}\sum_{s=0}^{n-km} c_{n, m}
 (pq^{k-1})^m q^s M(\pi_{n+2-s-km, s}) ,
$$
where $c_{n, m} := {n \choose m} -(k-1){n\choose m-1}$.
\end{Prop}

\begin{proof}
Let us code any trajectory from row $2$ to row $n+2$ in the tree $\tkab$ by its 
successive steps $(X_i)_{3\le i\le n+2}$ ($X_i=R$ for a right step or $L$ for a 
left step). 
Because of Proposition~\ref{Retour}, the label of the last edge of the 
trajectory does not change when we remove a pattern $RL^{k-1}$ and its 
weight is divided by $pq^{k-1}$.
By successively removing all patterns $RL^{k-1}$ in $(X_i)_{3\le i\le n}$, we obtain a \emph{reduced} sequence, that is a subsequence which never contains $(k-1)$ successive $L$'s, except eventually at the beginning. Moreover, its last edge has the same label as the last edge of the initial trajectory. Let us denote by $s$ the number of left steps at the beginning of the reduced sequence and by $m$ the number of removals. 
By removing these $s$ left steps, we obtain a trajectory in the tree $\rks$, whose last edge has the same label as the last edge initial trajectory and whose weight has been divided by $(pq^{k-1})^mq^{s}$.

Conversely, to a reduced sequence of length $n-km$ (for $0\le m\le \lfloor n/k\rfloor$), we can associate a set of trajectories in $\tkab$ ending in row $n+2$ by successively adding $m$ patterns $RL^{k-1}$. Denote by $c_{n,m}$ the cardinal of this set. 
To conclude the proof of the proposition, it remains to prove that $c_{n, m} := {n \choose m} -(k-1){n\choose m-1}$.
It is obviously true for all $n$ when $m=0$. 
We claim that the sequence $(c_{n,m})$ satisfies $c_{n+1, m}=c_{n,m}+c_{n,m-1}$ for any 
$n\ge 0$ and $1\le m\le \lfloor (n+1)/k\rfloor$. 
Indeed, there are two disjoint classes of trajectories in $\tkab$ of length $n+1$ which can be obtained from a reduced sequence by successively adding $m$ patterns $RL^{k-1}$: 
\begin{itemize}
\item Those for which no pattern $RL^{k-1}$ is inserted at the end. Their number is equal to $c_{n,m}$. Observe that this class is empty for $n=km$, since in this case, we start with an empty sequence.
\item Those for which at least one pattern $RL^{k-1}$ is inserted at the end. They can be obtained by first inserting a pattern at the end, then by inserting $(m-1)$ other patterns anywhere but at the end of the new sequence. Their number is equal to $c_{n,m-1}$.
\end{itemize}

It is straightforward to check that the numbers ${n \choose m} -(k-1){n\choose m-1}$ satisfy the same induction, thus that they coincide with $(c_{n,m})$.
\end{proof}

\begin{Rem} 
The coefficients $c_{n, m}$ can be obtained by generalizing Pascal's
triangle (which corresponds to $k=1$): We start by writing an infinite column of $1$. 
On line $k$, we write a second $1$ at the right of the first one. This
new $1$ is the beginning of a second column obtained by Pascal's rule : 
$c_{n+1,m}=c_{n,m}+c_{n,m-1}$.
On line $2k$, that is $k$ lines after the beginning of the second column, we start a third
column with the value $c_{2k,2}=c_{2k-1,1}$ and go on using Pascal's rule. 
Each new column starts $k$ rows after the previous one, its first term is given by the rule
$c_{pk,p}=c_{pk-1,p-1}$ (that is: the first term of a column is equal 
to the term one row and one column before it), and the next terms of
the column are given by Pascal's rule.

Here is the beginning of the triangle in the case $k=4$:
\[\begin{array}{cccccc}
	&m=0&m=1&m=2&m=3&m=4\\
n=0 &1 &&&&\\
n=1 & 1 &&&&\\
  &1&&&&\\
  & 1&&&&\\
n=k &  1 & 1 &&&\\
 &  1 & 2 &&&\\
 &  1 & 3 &&&\\
 &  1 & 4 &&&\\
n=2k &  1 & 5 & 4 &&\\
 &  1 & 6 & 9 &&\\
 &  1 & 7 & 15 &&\\
 &  1 & 8 & 22 &&\\
n=3k  & 1 & 9 & 30 & 22 &\\
 &  1 & 10 & 39 & 52 &\\
 & 1 & 11 & 49 & 91 &\\
  & 1 & 12 & 60 & 130 &\\
n=4k  & 1 & 13 & 72 & 190 & 130 \\
  & 1 & 14 & 85 & 162 & 320 \\
 & \vdots & \vdots & \vdots & \vdots & \vdots
\end{array}\]
%
%
\end{Rem}

\section{Proof of Theorem~\ref{Average1}}
\label{end}

Recall that we denoted by $(\beta_{j})$ the roots of the polynomial $P_k$ (studied in Lemma~\ref{Polyk}), with $\alpha_k=\beta_0$ the eigenvalue with largest modulus.
We deduce from~\eqref{equMpik} and Proposition~\ref{SuccNu} that there exists polynomials $Q_{j,s}$ (depending on $\ell_{s+1}$ and $\ell_{s+2}$) of degree less than the multiplicity of $\beta_j$, such that for any $n\ge 0$,
$$
M(\psi_{n+2}) = \sum_{j} \sum_{m=0}^{\lfloor n/k\rfloor} c_{n, m} \left(pq^{k-1}\right)^m \sum_{s=0}^{n-km}  q^s Q_{j,s}(n-km-s) \beta_{j}^{n-km-s} .
$$
We now study the contribution of each $\beta_j$: We will find an equivalent of the contribution of the dominant root $\beta_0$ and prove that the contribution of other roots is negligeable.

\begin{Lemme}\label{ctesbornees} 
For any $j$, the coefficients of the polynomials $(Q_{j, s})_s$ are uniformly bounded with respect to $s$. 
Moreover, the coefficients $(Q_{0, s})_s$ are positive.
\end{Lemme}
\begin{proof} This comes from Proposition~\ref{EllBornee} and Lemma~\ref{C0isPositive}.
\end{proof}

If $p>p_c$ (respectively if $p\le p_c$), let $\epsilon>0$ be small enough such that for all $j\ge 1$ (respectively $j\ge 3$), $\rho_j:=|\beta_j|(1+\epsilon)<\beta_0$.
The previous lemma implies that there exists a constant $K$ such that for all $n$, $|Q_{j,s}(n) \beta_{j}^{n}| \le K \rho_j^{n}$.
The contribution of $\beta_0$ to $M(\psi_{n+2})$ can be written as $U_{n,0}(\beta_0)$, where
$$
U_{n,0}(x) :=  x^{n} \sum_{m=0}^{\lfloor n/k\rfloor} c_{n, m} \left(\frac{pq^{k-1}}{x^{k}}\right)^m \sum_{s=0}^{n-km} Q_{0,s} \left(\frac{q}{x}\right)^s.
$$
The contribution of any other $\beta_j$ can be bounded by $K U_{n}(\rho_j)$, where
$$
U_{n}(x) :=  x^{n} \sum_{m=0}^{\lfloor n/k\rfloor} c_{n, m} \left(\frac{pq^{k-1}}{x^{k}}\right)^m \sum_{s=0}^{n-km}  \left(\frac{q}{x}\right)^s.
$$

Observe that $U_{n, 0}(x)$ is bounded by a constant times $U_{n}(x)$ since the coefficients $(Q_{0, s})_s$ are positive and bounded.

\begin{Prop}\label{Unj}
Let $x>0$. We set $f(x):=x\left(1+\frac{pq^{k-1}}{x^k}\right)$.
There exists $C(x)>0$, depending only on $x$, such that 
\begin{itemize}
\item If $x>q$ and $x^k>(k-1)pq^{k-1}$, then $U_{n,0}(x)\sim C(x) \bigl(f(x)\bigr)^n$ as $n\to\infty$;
\item If $x>q$ then $|U_{n}(x)|\le C(x) \bigl(f(x)\bigr)^n$;
\item $U_{n}(q) = O(n)$ as $n\to\infty$;
\item If $x<q$, then $U_{n}(x)=O(1)$  as $n\to\infty$.
\end{itemize}
\end{Prop}

\begin{proof}
\textbf{Case 1:} $x>q$. 
When $j=0$ (which corresponds to the dominant eigenvalue of $P_k$), the coefficients $(Q_{0, s})_s$ are positive and bounded, so
we can choose $S$ large enough such that 
$$
\left|\sum_{s=S+1}^{n-km} Q_{0,s} \left(\frac{q}{x}\right)^s \right|\le \delta.
$$
Set $A(x):=\sum_{s=0}^{S} Q_{0,s} \left(\frac{q}{x}\right)^s$.
Hence, 
$$
A(x)\ x^{n} \sum_{m=0}^{\lfloor (n-S)/k\rfloor} c_{n, m} \left(\frac{pq^{k-1}}{x^{k}}\right)^m
\le U_{n,0}(x)
\le (A(x) +\delta)\ x^{n} \sum_{m=0}^{\lfloor n/k\rfloor} c_{n, m} \left(\frac{pq^{k-1}}{x^{k}}\right)^m.
$$
Observe that, since $c_{n, m} = {n \choose m} -(k-1){n\choose m-1} = {n \choose m}\frac{n-km+1}{n-m+1}$, 
we can rewrite $\sum_{m=0}^{\lfloor n/k\rfloor} c_{n, m} \left(\frac{pq^{k-1}}{x^{k}}\right)^m$ as
$$
\left(1+\frac{pq^{k-1}}{x^k}\right)^{n} \E\left[\frac{1-kM/n+1/n}{1-M/n+1/n} \mathbbmss{1}_{\left\{M\le \lfloor n/k\rfloor\right\}}\right],
$$
where $M$ is a binomial random variable with parameters $(n, \theta)$. 
Since $M/n\to \theta$ almost surely as $n\to\infty$, we get that this expectation goes to $\frac{1-k\theta}{1-\theta}$ if $\theta< 1/k$ (that is,  if $x^k> (k-1)pq^{k-1}$), and the same is true if we replace $n$ by $n-S$. 
Therefore, if $x^k>(k-1)pq^{k-1}$, for $n$ large enough, 
$$
A(x)\left(\frac{1-k\theta}{1-\theta} -\delta\right)\le \frac{ U_{n,0}(x)}{x^{n}\left(1+\frac{pq^{k-1}}{x^k}\right)^{n}}\le (A(x) +\delta)\left(\frac{1-k\theta}{1-\theta}+\delta\right).
$$

For $j\not=0$, since $x>q$, $|U_{n}(x)|$ is bounded above, up to a multiplicative constant $ C(x)$, by
$$
x^{n} \sum_{m=0}^{\lfloor n/k\rfloor} c_{n, m} \left(\frac{pq^{k-1}}{x^{k}}\right)^m 
\le x^{n}\left(1+\frac{pq^{k-1}}{x^k}\right)^{n},
$$
because $c_{n, m}\le {n\choose m}$.

\textbf{Case 2:} $x=q$. We easily see that 
$$
U_{n}(q) = \sum_{m=0}^{\lfloor n/k\rfloor} c_{n, m} p^nq^{n-m} \ (n-km+1) \le n.
$$

\textbf{Case 3:} $x<q$. 
$$
U_{n}(x) = x^{n} \sum_{m=0}^{\lfloor n/k\rfloor}\sum_{s=0}^{n-km} c_{n, m} \left(\frac{pq^{k-1}}{x^{k}}\right)^m \frac{(q/x)^{n-km+1}-1}{q/x-1},
$$
which is, up to a multiplicative constant $C(x)$, less than
$$
\sum_{m=0}^{\lfloor n/k\rfloor}c_{n, m} p^m q^{n-m}
\le 1.
$$
\end{proof}

\begin{proof}[Proof of Theorem~\ref{Average1}]
Assume $p>p_c$. 
We know from Lemma~\ref{Polyk} that $\beta_0=\alpha_k >q$, and all the other roots $\beta_j$ are such that $|\beta_j|<\alpha_k$. 

We use Proposition~\ref{Unj}, and need to understand the variations of $f$. Elementary computations show that $f(x)$ decreases when $x$ ranges from $0$ to $x_{\mbox{\scriptsize min}}:=\sqrt[k]{(k-1)pq^{k-1}}$, then increases. Observe that $f(q)=1$, hence $f(x_{\mbox{\scriptsize min}})\le1$.

We claim that $\alpha_k>x_{\mbox{\scriptsize min}}$. 
This is true if $p\le 1/k$ because, in this case, $q\ge x_{\mbox{\scriptsize min}}$. 
This remains true if $p>1/k$: Otherwise, we would have $q<\alpha_k\le x_{\mbox{\scriptsize min}}$, hence $f(\alpha_k)<1$. Moreover, if $|\beta_j|>q$, we also have $f(|\beta_j|)<1$, and the contributions of other $\beta_j$'s is, by Proposition~\ref{Unj}, at most linearly increasing with $n$. This would imply that $M(\psi_n)=O(n)$. But we know from~\cite{janvresse2008b} that when $p>1/k$, the $n$-th term of a $(p,\lambda_{k})$-random Fibonacci sequence almost surely grows exponentially fast. By Jensen's inequality, this is all the more true for its expected value $ M(\psi_n)$, so we get a contradiction.

It follows from Proposition~\ref{Unj} that the contribution to $M(\psi_{n+2})$ of $\beta_0$ is 
$U_{n,0}(\beta_0)\sim C(\beta_0) \bigl(f(\beta_0)\bigr)^n$ as $n\to\infty$, and that for $j\not=0$, the contribution of $\beta_j$, which is bounded by $KU_{n}(\rho_j)$, is negligeable with respect to $U_{n,0}(\beta_0)$. This ends the proof of 
Theorem~\ref{Average1} in the case $p>p_c$. 

Assume that $p=p_c$. We know from Lemma~\ref{Polyk} that $\beta_0 = q$, and all the other roots $\beta_j$ are such that $|\beta_j|\le q$. We thus deduce from Proposition~\ref{Unj} that $M(\psi_{n})$ grows at most linearly.

Assume $p<p_c$. 
By Lemma~\ref{C0isPositive}, we know that $(Q_{1, s})_s$ and $(Q_{2, s})_s$ are null. 
Moreover, we know from Lemma~\ref{Polyk} that $\beta_0 < q$ and that $\beta_j<\beta_0$ for all $j\ge 3$.
Using Proposition~\ref{Unj}, we conclude that $M(\psi_{n})$ is bounded. 
\end{proof}

\section{Non-analyticity in the neighbourhood of $2$}
\label{Sec:proof_of_cor}

\begin{proof}[Proof of Corollary~\ref{Cor:analyticite_p}]
By Theorem~\ref{lambda_grand}, we know that for $\lambda\ge2$, ${\mathcal{G}}(\lambda)$ is a root of the polynomial $Q_{\lambda}(X) := X^2-\lambda X -(2p-1)$. 
On the other hand, Theorem~\ref{Average1} says that for $\lambda=\lambda_{k}$ and $p>p_c$,  ${\mathcal{G}}(\lambda_k)=\alpha_k(p)\left[1+\frac{pq^{k-1}}{\alpha_k(p)^k}\right]$, where $\alpha_k(p)$ is a positive root of $P_k$. 
Thus, for any $k\ge 3$, we easily get that 
$$
\alpha_k(p)^{2k-2}\ Q_{\lambda_k}\left({\mathcal{G}}(\lambda_k)\right) 
= 2pq^{k-1}\left(\alpha^k(p) + pq^{k-1}\right) > 0, 
$$
which proves that, for any $k\ge 3$, ${\mathcal{G}}(\lambda_k)$ is not a root of $Q_{\lambda_k}$.
\end{proof}

\begin{proof}[Proof of Corollary~\ref{Cor:analyticite_1/2}]
The growth rate of the expected value of a $(1/2,\lambda)$-random Fibonacci sequence has a very simple expression: It is equal to $\lambda$ when $\lambda\ge2$ (Theorem~\ref{lambda_grand}) and to $2\alpha_k - \lambda_k$, where $\alpha_k$ is the only positive root of the polynomial $X^k-\lambda_{k}X^{k-1}-1/2^k$, when $\lambda=\lambda_{k}$.
Indeed, in the case $p=1/2$, the polynomial $P_k(X)$ in Theorem~\ref{Average1} can be rewritten as $(X^k+1/2^k)Q_k(X)$, where $Q_k(X):=X^k-\lambda_kX^{k-1}-2^{-k}$. 
Moreover, $Q_k(\alpha_k) = 0$ implies that 
$\alpha_k\left(1+\frac{1}{2^k\alpha_k^k}\right)=2\alpha_k - \lambda_k$.

Observe that $Q_{k}(\lambda_k)<0$, which obviously implies that $\alpha_{k}> \lambda_k$.
Since $Q_{k}(\alpha_{k})=0$, we have $\alpha_{k}^{k-1}(\alpha_{k}- \lambda_{k}) = 1/2^{k}$. 
Thus, $\alpha_{k}>\lambda_{k}\ge 1$ proves that 
$0<\alpha_{k}- \lambda_{k}< 2^{-k} $.

Since $\lambda_{k}$ tends to $\lambda_{\infty}=2$ when $k$ goes to
infinity, if ${\mathcal{G}}'(2)$ exists, then we must have
\[{\mathcal{G}}'(2) = 
\lim_{k\rightarrow +\infty} \left(\frac{{\mathcal{G}}(\lambda_{k}) - {\mathcal{G}}(2)}{\lambda_{k}-2}\right) = 
1 + \lim_{k\rightarrow +\infty} 2 \frac{\alpha_{k}-\lambda_{k}}{\lambda_{k}-2}.\]

The numerator of the latter expression tends exponentially fast to $0$, whereas the denominator is equivalent to
$2\pi^2/k^2$, so we get ${\mathcal{G}}'(2)=1$.

If ${\mathcal{G}}$ is of class $C^2$ at $\lambda=2$, let us write its
Taylor expansion at order $2$:
\[{\mathcal{G}}(\lambda_{k})={\mathcal{G}}(2)+(\lambda_{k}-2){\mathcal{G}}'(2)+
\frac{(\lambda_{k}-2)^2}{2!}{\mathcal{G}}''(2)+O((\lambda_{k}-2)^3).\]

We then get that ${\mathcal{G}}''(2)$ is equal to the limit of the ratio
$2!\,2(\alpha_{k}-\lambda_{k})/(\lambda_{k}-2)^2$, which is equal to $0$. 
The nullity of the $n$-th derivative of $\mathcal{G}$ is obtained in a similar way, by an induction argument.

Hence, provided ${\mathcal{G}}$ is of class $C^\infty$ at $\lambda=2$, 
${\mathcal{G}}(2)=2$, ${\mathcal{G}}'(2) = 1$ and ${\mathcal{G}}^{(n)}(2) = 0$ for any $n\ge 2$.
The only possibility for ${\mathcal{G}}$ to be analytic at $\lambda=2$ is to satisfy
 ${\mathcal{G}}(\lambda)=\lambda$ on a neighbourhood of $2$. 
But we also have ${\mathcal{G}}(\lambda_k)=2\alpha_{k}-\lambda_k$, which would imply 
$\alpha_{k}=\lambda_k$ for $k$ large enough. 
This would contradict $P_k(\lambda_{k})=-2^{-k}<0$.
\end{proof}

\section{Open questions}
\label{Sec:open}

\subsection{Critical value}
Theorem~\ref{Average1} states that for $p=p_c$, the growth of $\E(g_{n})$ is at most linear. 
The proof of this result uses the fact that the labels $\ell_n$ on the leftmost branch of $\tkab$ are bounded. 
It is proved in~\cite{janvresse2008b} that if $a/b$ admits a finite $\lambda$-continued fraction expansion, the sequence $(\ell_n)$ is ultimately periodic. The arguments developed in the proof of Proposition~\ref{Unj} show that, in this case, $\E(g_{n})$ does grow linearly.
However, we believe that for most choices of the initial values $a$ and $b$, these labels decrease exponentially fast to zero, and that this ensures that $\E(g_{n})$ is bounded.

\subsection{Numerical simulation}
As suggested to us by Steven Finch, from INRIA, numerical evidence of the growth rate of the expected value of a random Fibonacci sequence is not easy to obtain. 
This is due to the different behaviour of $g_n$ and of $\E(g_n)$. 
For $\lambda=\lambda_k$, comparison with the result obtained in~\cite{janvresse2008b} shows that for $\frac{2-\lambda_k}{4}<p\le1/k$, the expected value of the $n$-th term of a random Fibonacci sequence increases exponentially fast, whereas the sequence contains  almost-surely a bounded subsequence. 
When $1/k<p<1$, numerical estimation of the growth rate of the expected value of $F_n$ given by Theorem~\ref{Average1} suggests that it
is strictly greater than the almost-sure growth rate. 
This would imply that the variance of $g_n$ increases exponentially fast with growth rate at least twice the growth rate of the expected value. 
In~\cite{Makover2006} where the case $p=1/2$ and $\lambda=1$ is considered, the growth rate of the variance is proved to be equal to $1+\sqrt{5}$. 
It would be of interest to know better about the exact value of the variance for any $p$ and, more generally, about the moments of higher order.

\subsection{Generalization}
\subsubsection{Linear case}
It is also of interest to consider the {\em linear case}, given by the relation $g_{n+1}= \lambda g_{n}\pm g_{n-1}$, 
where the signs are given by an i.i.d. sequence of Bernoulli random variables of parameter $p$. 
We then have the easy induction
$\E(g_{n}) = \lambda \E(g_{n-1}) + (2p-1)\E(g_{n-2})$ for any $\lambda$.
The corresponding polynomial has two real roots $(\lambda\pm \sqrt{\lambda^2+4(2p-1)})/2$, and we easily get an explicit expression of $\E(g_{n})$ depending on the initial values. The question of interest in this setting would be to study the exponential growth of $\E(|g_{n}|)$.
The almost-sure growth rate of $|g_n|$ in the linear case is studied in~\cite{janvresse2008b} for $\lambda=\lambda_k$ ($k\ge3$) and $\lambda\ge2$, and turns out to be more difficult than in the non-linear case. 
The analysis of $\E(|g_{n}|)$ is also more intricate. Although we can embed the tree $\rkab$ in the tree of all possible sequences, this embedding is more complex than what we describe in the present article: 
The main reason is that after the removal of a pattern $RL^{k-1}$, left and right children are exchanged.
However we think our method can be adapted to the linear case.

\subsubsection{Underlying structure}
In the present paper, the underlying probabilistic structure is a Bernoulli scheme of parameter $p$. 
There is no significant doubt that our method extends without significantly new ideas to some more general processes, as for example the one in which the choice of the plus or minus sign is given by two coins alternatively tossed, the first one with parameter $p$ and the other one with parameter $p'$.

We may also investigate what happens with a deterministic rule, like the codage of an irrational rotation in the circle. This leads to some interesting constructions which involve substitutions. These are to be explained in a forthcoming paper.


\bibliography{rf-rosen.bib}

\providecommand{\bysame}{\leavevmode\hbox to3em{\hrulefill}\thinspace}
\providecommand{\MR}{\relax\ifhmode\unskip\space\fi MR }
\providecommand{\MRhref}[2]{%
  \href{http://www.ams.org/mathscinet-getitem?mr=#1}{#2}
}
\providecommand{\href}[2]{#2}
\begin{thebibliography}{1}

\bibitem{hecke1936}
E.~Hecke, \emph{\"{U}ber die {B}estimmung {D}irichletscher {R}eihen durch ihre
  {F}unktionalgleichung}, Math. Ann. \textbf{112} (1936), no.~1, 664--699.

\bibitem{janvresse2008b}
\'Elise Janvresse, Beno\^it Rittaud, and Thierry de~la Rue, \emph{Almost-sure
  growth rate of generalized random {F}ibonacci sequences}, Preprint, 2008.

\bibitem{janvresse2007}
\bysame, \emph{How do random {F}ibonacci sequences grow?}, to appear in Prob.
  Th. Rel. Fields, 2008.

\bibitem{Makover2006}
Eran Makover and Jeffrey McGowan, \emph{An elementary proof that random
  {F}ibonacci sequences grow exponentially}, J. Number Theory \textbf{121}
  (2006), no.~1, 40--44.

\bibitem{rittaud2006}
Beno\^it Rittaud, \emph{On the average growth of random {F}ibonacci sequences},
  J. Int. Seq. \textbf{10} (2007), no.~07.2.4.

\bibitem{rosen1954}
David Rosen, \emph{A class of continued fractions associated with certain
  properly discontinuous groups}, Duke Math. J. \textbf{21} (1954), 549--563.

\bibitem{viswanath2000}
Divakar Viswanath, \emph{Random {F}ibonacci sequences and the number
  {$1.13198824\ldots$}}, Math. Comp. \textbf{69} (2000), no.~231, 1131--1155.

\end{thebibliography}

\end{document}